\documentclass{amsart}

\usepackage{graphicx}
\usepackage{amsthm}
\usepackage{subcaption}
\newsavebox\myboxA
\newsavebox\myboxB
\newlength\mylenA

\newcommand*\xoverline[2][0.75]{%
    \sbox{\myboxA}{$\m@th#2$}%
    \setbox\myboxB\null
    \ht\myboxB=\ht\myboxA%
    \dp\myboxB=\dp\myboxA%
    \wd\myboxB=#1\wd\myboxA
    \sbox\myboxB{$\m@th\overline{\copy\myboxB}$}
    \setlength\mylenA{\the\wd\myboxA}
    \addtolength\mylenA{-\the\wd\myboxB}%
    \ifdim\wd\myboxB<\wd\myboxA%
       \rlap{\hskip 0.5\mylenA\usebox\myboxB}{\usebox\myboxA}%
    \else
        \hskip -0.5\mylenA\rlap{\usebox\myboxA}{\hskip 0.5\mylenA\usebox\myboxB}%
    \fi}
\makeatother
\newcommand{\lowerromannumeral}[1]{\romannumeral#1\relax}

\newtheorem{theorem}{Theorem}[section]
\newtheorem{lemma}[theorem]{Lemma}
\newtheorem{cor}[theorem]{Corollary}
\theoremstyle{definition}
\newtheorem{definition}[theorem]{Definition}

\theoremstyle{remark}
\newtheorem{remark}[theorem]{Remark}

\newtheorem{claim}{Claim}

\newtheorem{case*}{Case}

\makeatletter
\@addtoreset{claim}{theorem}
\makeatother

\numberwithin{equation}{section}

\begin{document}

\title{Link diagrams with low Turaev genus}

\author{Seungwon Kim}
\address{Department of Mathematics, The Graduate Center, CUNY}
\email{skim2@gradcenter.cuny.edu}

\maketitle
\begin{abstract}
We classify link diagrams with Turaev genus one and two in terms of an alternating tangle structure of the link diagram. The proof involves surgery along simple loops on the Turaev surface, called cutting loops, which have corresponding cutting arcs that are visible on the planar link diagram.  These also provide new obstructions for a link diagram on a surface to come from the Turaev surface algorithm. We also show that inadequate Turaev genus one links are almost-alternating.
\end{abstract}

\section{Introduction}
In \cite{Tu1}, Turaev introduced the Turaev surface of a link diagram to use in a new proof of Tait's conjecture. Its minimal genus among all possible diagrams of a given link is a link invariant, which is called the Turaev genus of the link.
The Turaev surface is a Heegaard surface of $S^3$, such that the link diagram on its Turaev surface is alternating, and its projection divides the surface into a disjoint union of discs. These properties imply that the Turaev genus measures how far a given link is from being alternating (see \cite{CK,DFKLS}).

In this paper, we classify link diagrams with low Turaev genus in terms of an alternating tangle structure on the link diagram. An alternating tangle structure on a diagram $D$ on $S^2$ provides a decomposition of $D$ into maximally connected alternating tangles, defined by Thistlethwaite \cite{T1}, and below in Section \ref{definitions}. Let $g_T(D)$ denote the Turaev genus of $D$, which is defined in Section 2 below.

Our main results are the following:
\begin{theorem}\label{Turaev1}
Every prime connected link diagram $D$ on $S^2$ with $g_T (D)=1$ is a cycle of alternating $2$-tangles, as shown in the figure below. 
\end{theorem}

\centerline{\includegraphics[height=0.5in]{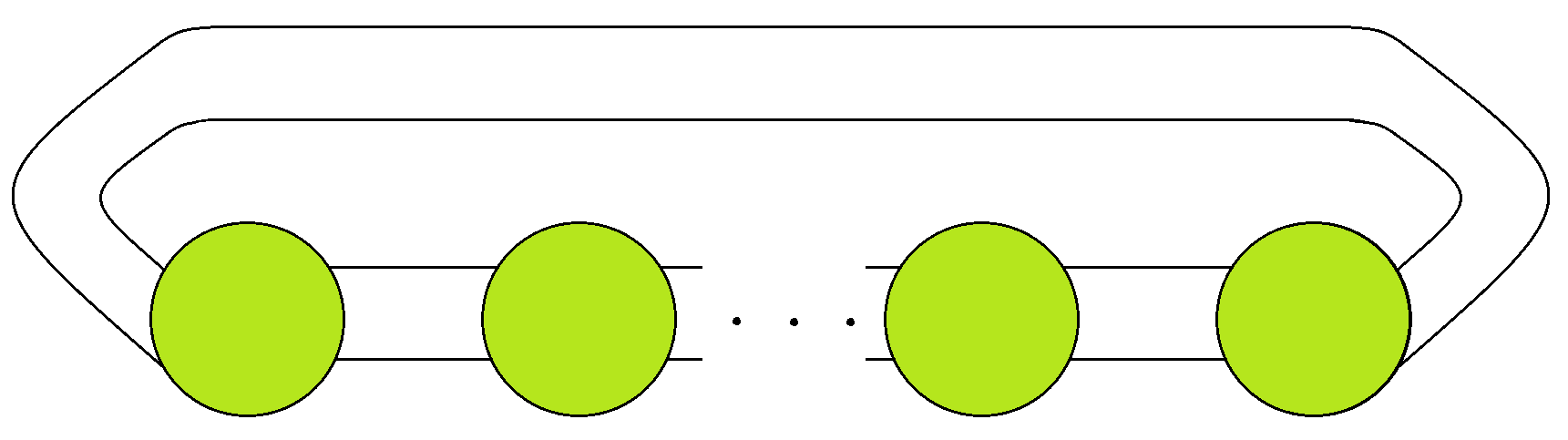}}

\begin{theorem}\label{Turaev2}
Every prime connected link diagram $D$ on $S^2$ with $g_T(D)=2$ has one of the eight alternating tangle structures shown below in Figure \ref{genus2}.
\end{theorem}

Green discs represent maximally connected alternating tangles, and black arcs are non-alternating edges of $D$.
In Figure \ref{genus2}, ribbons denote an even number of linearly connected alternating $2$-tangles: 
\includegraphics[height=0.15in]{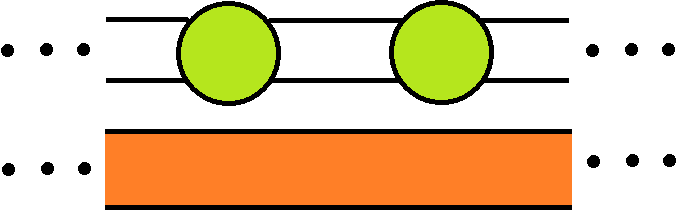}

Armond and Lowrance \cite{AL} proved a similar classification independently at the same time. They classified link diagrams with Turaev genus one and two in terms of their alternating decomposition graphs upto graph isomorphism. While their proof is primarily combinatorial, our proof is primarily geometric. Our result is also somewhat stronger; we classify all possible embeddings of alternating decomposition graphs into $S^2$. 
Their graphs can be obtained from our Figure \ref{genus2} simply by erasing the colors from the ribbons, and contracting the boundaries of the alternating tangles into vertices. Our cases 1, 3, 6 give their case 2, our cases 2, 5 give their case 3, and the other cases correspond bijectively, with our cases 4, 7, 8 giving their cases 1, 4, 5 respectively.
\begin{figure}[h]
 
\begin{subfigure}{0.3\textwidth}
\centering
\includegraphics[height=1cm]{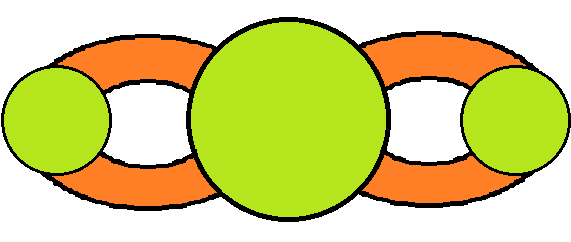} 
\caption{Case 1}
\label{case1}
\end{subfigure}
\begin{subfigure}{0.3\textwidth}
\centering
\includegraphics[height=1cm]{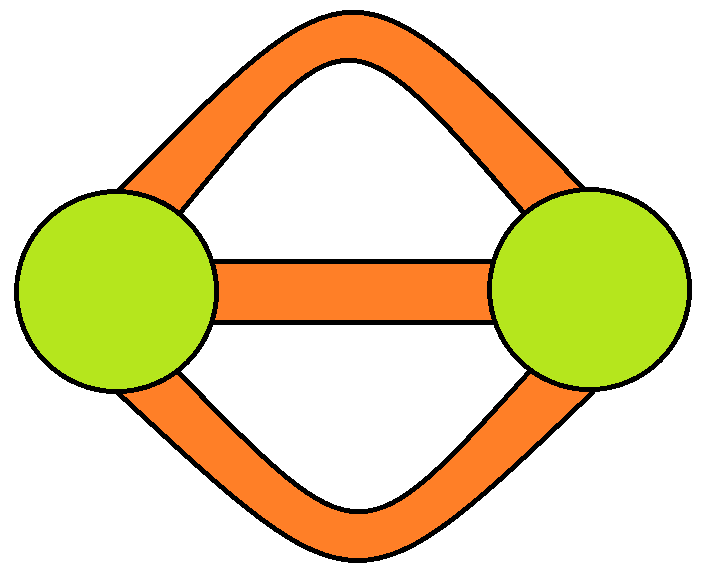}
\caption{Case 2}
\label{case2}
\end{subfigure}
\begin{subfigure}{0.3\textwidth}
\centering
\includegraphics[height=1cm]{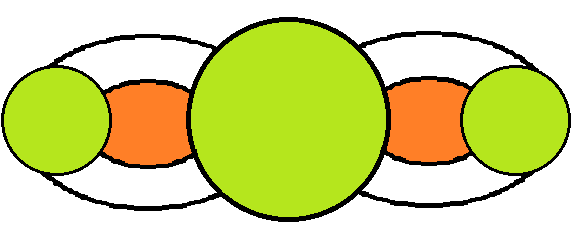}
\caption{Case 3}
\label{case3}
\end{subfigure}
\begin{subfigure}{0.3\textwidth}
\centering
\includegraphics[height=1cm]{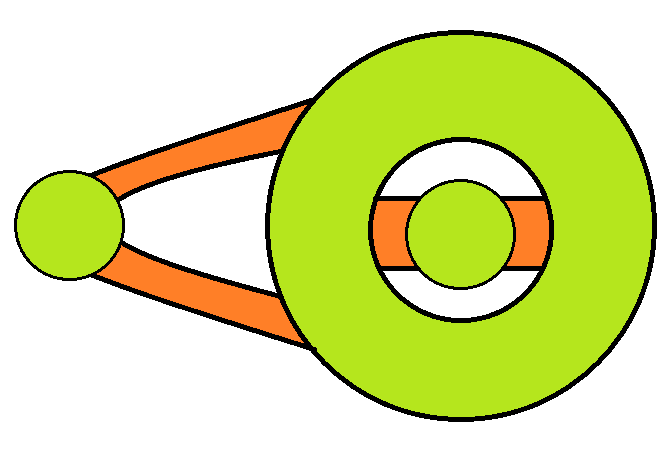}
\caption{Case 4}
\label{case4}
\end{subfigure}
\begin{subfigure}{0.3\textwidth}
\centering
\includegraphics[height=1cm]{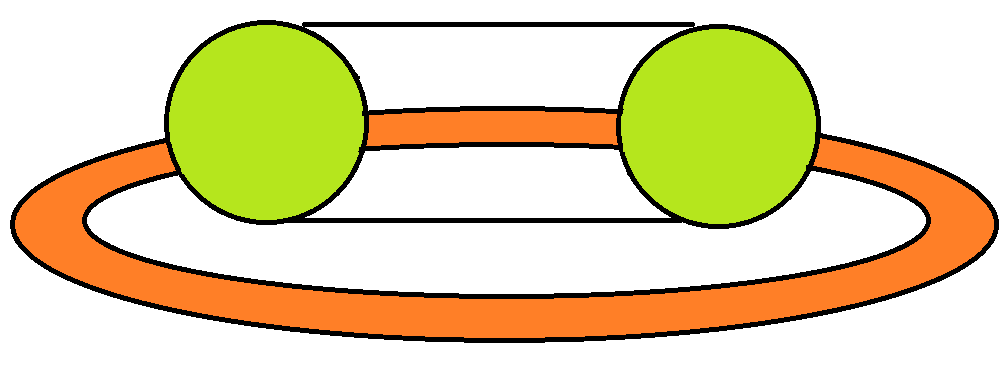}
\caption{Case 5}
\label{case5}
\end{subfigure}
\begin{subfigure}{0.3\textwidth}
\centering
\includegraphics[height=1cm]{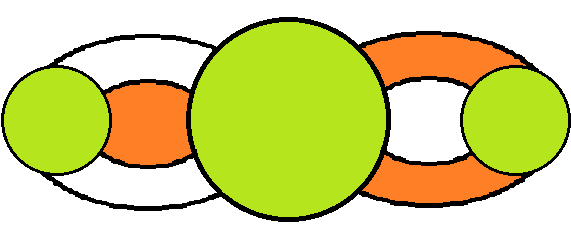}
\caption{Case 6}
\label{case6}
\end{subfigure}
\begin{subfigure}{0.3\textwidth}
\centering
\includegraphics[height=1.5cm]{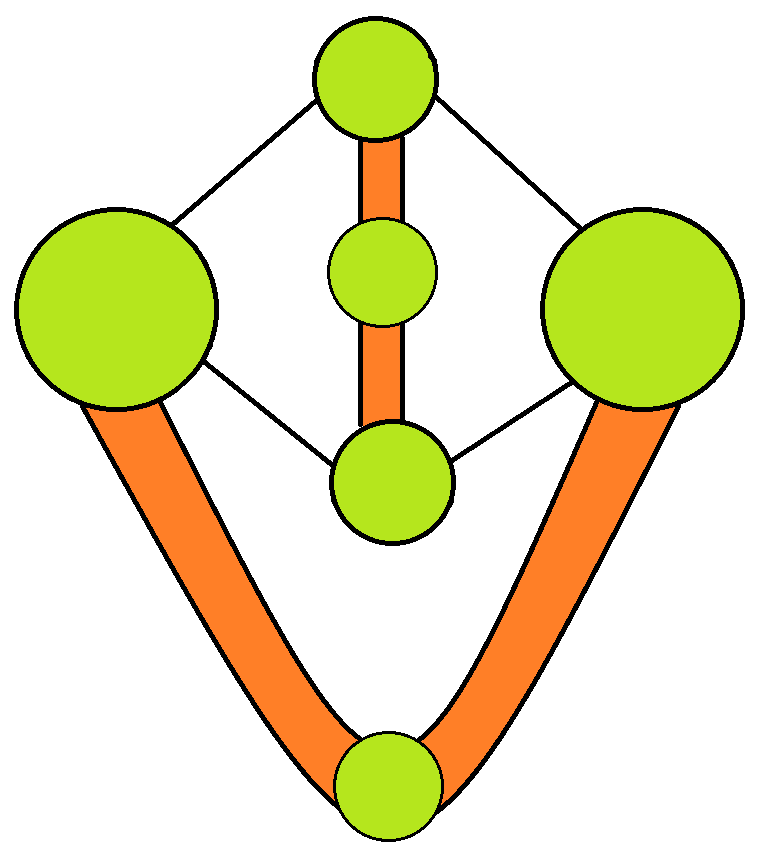}
\caption{Case 7}
\label{case7}
\end{subfigure}
\begin{subfigure}{0.3\textwidth}
\centering
\includegraphics[height=1.5cm]{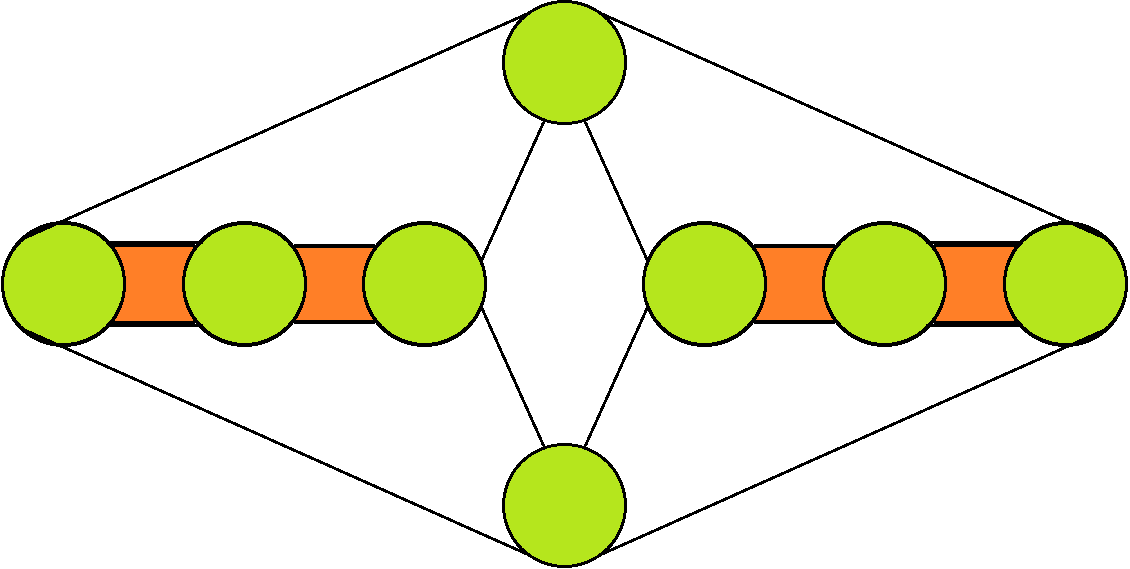}
\caption{Case 8}
\label{case8}
\end{subfigure} 
\caption{}
\label{genus2}
\end{figure}

A non trivial link diagram is \textit{almost-alternating} if one crossing change makes the diagram alternating. A non trivial link is \textit{almost-alternating} if it admits an almost-alternating diagram. It is conjectured that all Turaev genus one links are almost alternating. This conjecture has been proved for non-alternating Montesinos links, and semi-alternating links \cite{A1,AK,L1}. We prove this conjecture for inadequate links using our new geometric methods.

\begin{theorem}\label{almostalternating}
Let $L$ be an inadequate non-split prime link with $g_T(L)=1$. Then $L$ is almost-alternating.
\end{theorem}
\subsection{\bf{Acknowledgements}}
The author thanks Ilya Kofman and Abhijit Champanerkar for helpful comments and guidance.
\section{Turaev genus}
A link diagram $D$ on a surface $F$ is a projection of a link $L$ onto $F$, which is a $4$-valent graph on $F$ such that each vertex is identified as an over or under crossing of $D$. For each crossing in $D$, we put a crossing ball so that $L$ lies on $F$ except near crossings of $D$, where $L$ lies on a crossing ball as shown in the figure below (see \cite{M1,ADK}). 

\centerline{\includegraphics[scale=0.2]{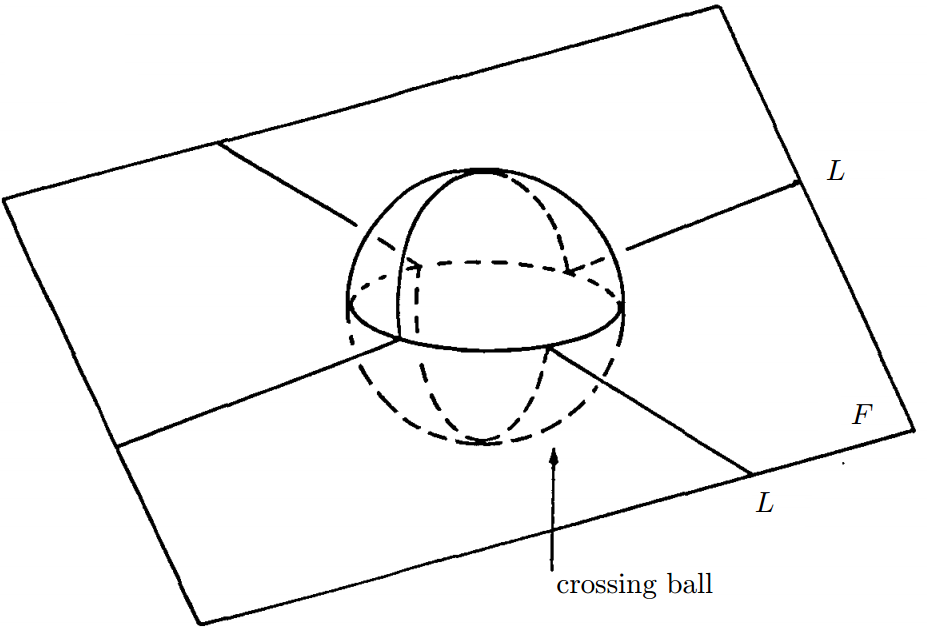}}

According to \cite{ADK}, we call this a crossing ball configuration of the link $L$ corresponding to the diagram $D$. With this configuration, we can obtain the $A$-smoothing and the $B$-smoothing as shown: \includegraphics[height=0.3in]{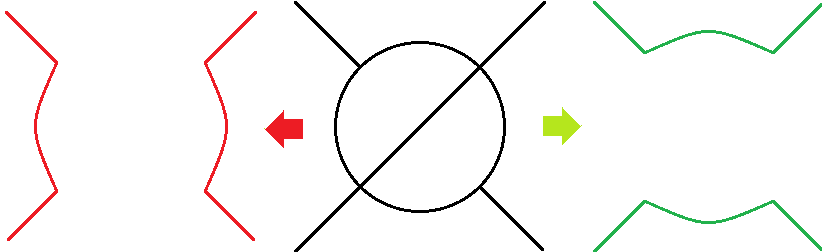}

A state $s$ of $D$ on $S^2$ is a choice of smoothing at every crossing, resulting in a disjoint union of circles on $S^2$. Let $|s|$ denote the number of circles in $s$. Let $s_A$ denote the all-$A$ state, for which every crossing of $D$ is replaced by an $A$-smoothing. Similarly, $s_B$ is the all-$B$ state of $D$.

Now, as we push $s_A$ up and $s_B$ down, then each state circle sweeps out an annulus. We can glue all such annuli and equatorial discs of each crossing ball to get a cobordism between $s_A$ and $s_B$. Note that each equatorial disc is a saddle of the cobordism.

For any link diagram $D$, the Turaev surface $F(D)$ is obtained by attaching $|s_A| + |s_B|$ discs to
all boundary circles of the cobordism above. Note that the crossing ball configuration of $D$ on $S^2$ induces a crossing ball configuration of $D$ on $F(D)$, hence, we can also consider $D$ as a link diagram on $F(D)$.

The Turaev genus of $D$ is defined by
\begin{equation}
g_T (D)=g(F(D))=(c(D)+2 - |s_A| - |s_B|)/2.
\end{equation}

The Turaev genus of any non-split link $L$ is defined by
\begin{equation}
g_T (L) = min\{\, g_T (D) \: | \: \text{$D$ is a diagram of $L$}\} .
\end{equation}

The properties below follow easily from the definitions (see \cite{CK}).

\begin{enumerate}
\renewcommand{\labelenumi}{(\roman{enumi})}
\item $F(D)$ is an unknotted closed orientable surface in $S^3$; i.e., $S^3 - F(D)$ is a disjoint
union of two handlebodies.
\item $D$ is alternating on $F(D)$.
\item $L$ is alternating if and only if $g_T (L) = 0$, and if $D$ is an alternating diagram then $F(D) = S^2$.
\item $D$ gives a cell decomposition of $F(D)$,
for which the 2-cells can be checkerboard colored on $F(D)$, with discs corresponding
to $s_A$ and $s_B$ respectively colored white and black.
\item This cell decomposition is a Morse decomposition of $F(D)$, for which $D$ and the
crossing saddles are at height zero, and the $|s_A|$ and $|s_B|$ 2-cells are the maxima and
minima, respectively.
\end{enumerate}

We will say that a link diagram $D$ on a surface $F$ is cellularly embedded if $F-D$ consists of open discs.


\section{Definitions}\label{definitions}
Throughout this paper, let $D$ be a connected link diagram on $S^2$ which is checkerboard colored. 
An edge of $D$, joining two crossings of $D$, is \textit{alternating} if one end is an underpass and the other end an overpass. Otherwise, an edge is \textit{non-alternating}. 
$D$ is \textit{prime} if every simple loop on $S^2 - \{\text{crossings}\}$ which intersects $D$ in two points bounds a disc on $S^2$ which does not have any crossings inside. Otherwise, D is said to be \textit{composite} and any such simple loop that has crossings on both sides is called a \textit{composite circle} of $D$.
We will say that a crossing of $D$ is positive or negative, respectively, as shown: 
\includegraphics[height=0.2in]{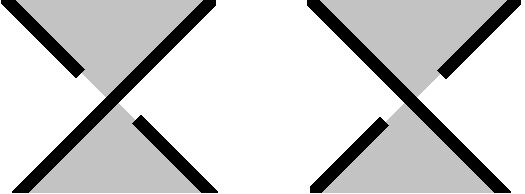}
In each alternating tangle all crossings have the same sign, so the tangle is either positive or negative.

An \textit{alternating tangle structure on a diagram $D$} \cite{T1} is defined as follows. For every non-alternating edge of $D$, take two points in the interior. Inside each face of $D$ containing non-alternating edges, pairs of such points are to be joined by disjoint arcs in the following way: Every arc joins two adjacent points on the boundary of the face, and these points are not on the same edge of $D$. Then the union $\Gamma$ of every arc is a disjoint set of simple loops on $S^2$. Let $\Delta$ be the closure of one of the components of $S^2 - \Gamma$ containing at least one crossing of $D$, then each edge of $D$ entirely contained in $\Delta$ is alternating. 

We will call the pair $(\Delta, \Delta \cap D)$ a \textit{maximally connected alternating tangle} of $D$. Let $n$ be the number of all the maximally connected alternating tangles of $D$. We will call ($D$, $\Delta_1 \cap D$, $\Delta_2 \cap D$, $\cdots$, $\Delta_n \cap D$) an \textit{alternating tangle structure} of $D$ and the closure of a component of $S^2 - \{\Delta_1, \cdots \Delta_n\}$ a \textit{channel region} of $D$. 



An alternating tangle structure of $D$ is a \textit{cycle of alternating 2-tangles} if it satisfies the following properties :
\begin{enumerate}
\renewcommand{\labelenumi}{(\roman{enumi})}
\item Every maximally connected alternating tangle of $D$ is a pair of a disc and an alternating $2$-tangle,
\item Any pair of maximally connected alternating tangles is connected with either two arcs or zero arcs in the channel region.
\end{enumerate}


Our key tools are the cutting loop and the cutting arc. As defined below, a cutting loop is a simple loop on the Turaev surface which is a topological obstruction for a given Heegaard surface with an alternating diagram on it to be the Turaev surface. A cutting arc is a simple arc on $S^2$ which is used to identify a cutting loop.

Let $D$ be a prime diagram. We can isotope $s_A$ and $s_B$ so that $s_A \cap s_B \cap D = $ \{midpoints of non-alternating edges of $D$\}. A \textit{cutting arc} $\delta$ is a simple arc in $S^2$ such that $\partial \delta = \delta \cap D \cap \alpha \cap \beta$ for a state circle $\alpha \subset s_A$ and another state circle $\beta \subset s_B$. See the figure below for example.  

\centerline{\includegraphics[height=1in]{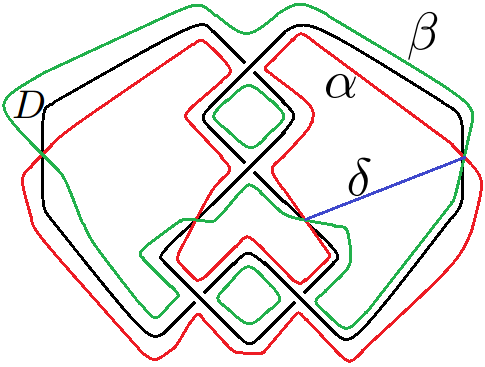}}

A \textit{cutting loop} $\gamma$ of a prime non-alternating diagram $D$ is a simple loop on $F(D)$ satisfying the following properties 
\textup{:}
\begin{enumerate}
\item $\gamma$ is non-separating on $F(D)$,
\item $\gamma$ intersects $D$ twice in $F(D)-\{\text{equatorial discs}\}$,
\item $\gamma$ bounds a disc $U_\gamma$ in one of the handlebodies bounded by $F(D)$ such that $U_\gamma \cap S^2$ is a cutting arc $\delta$.  
The disc $U_\gamma$ is called a \textit{cutting disc} of $D$.
\end{enumerate}

Every cutting loop has a corresponding cutting arc. We will prove the converse in Theorem \ref{surgerytheorem} below.

Let $\tau$ be a simple arc on $S^2 - \{\text{crossings}\}$ such that $\partial \tau = \tau \cap D$. A surgery along $\tau$ is the procedure of constructing a new link diagram $D'$ as follows:
\begin{equation}
D' = (D-(\partial \tau \times [-\epsilon,\epsilon])) \cup (\tau \times \{-\epsilon, \epsilon\}).
\end{equation}

\centerline{\includegraphics[scale=0.25]{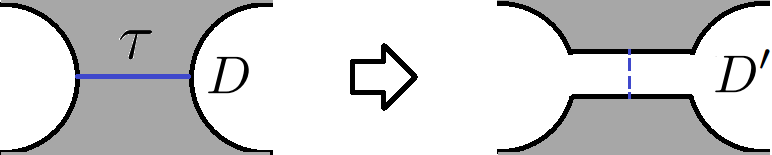}}

Let $\gamma$ be a cutting loop of $D$. A surgery along $\gamma$ is the procedure of constructing a new surface  $F'(D)$ as follows
\textup{:}
\begin{equation}
F'(D) = (F(D)-(\gamma \times [-\epsilon,\epsilon]))  \cup (U_\gamma \times \{-\epsilon, \epsilon\})
\end{equation}
and constructing a new diagram $D'$ both on $F'(D)$ and on $S^2$
\begin{equation}
D' = (D-(\partial \delta \times [-\epsilon,\epsilon])) \cup (\delta \times \{-\epsilon,\epsilon\})
\end{equation}

More generally, a surgery along any simple loop $\gamma$ on $F(D)-\{\text{equatorial discs}\}$ can be defined similarly if $\gamma$ satisfies conditions (2) and (3) in the definition of cutting loops, with $U_\gamma \cap S^2 = \tau$, where $\tau$ is a simple arc as above. See the figure below for example.

\centerline{\includegraphics[scale=0.2]{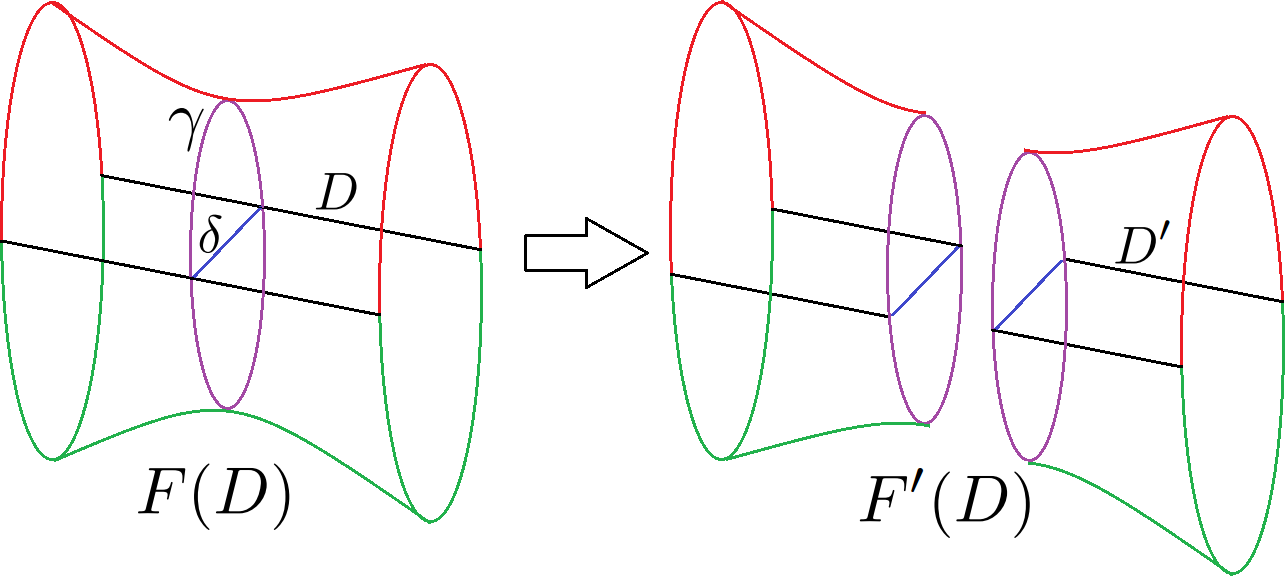}}

\section{Classification of Turaev genus one diagrams}
\begin{theorem}\label{surgerytheorem}
If $D$ is a prime non-alternating diagram then there exists a cutting arc $\delta$. Moreover, every cutting arc $\delta$ determines a corresponding cutting loop $\gamma$ on $F(D)$. After surgery along $\delta$ and $\gamma$, we get $F(D')=F'(D)$ and $g_T(D')=g_T(D)-1$.
\end{theorem}

\begin{proof}

First, we show the existence of a cutting arc. Consider a state circle $\alpha \subset s_A$ such that $\alpha \cap s_B \neq \emptyset$. Take the outermost bigon in the disc bounded by $\alpha$ which is formed by $\alpha$ and $s_B$. Near this bigon, we have two possible configurations of $D$, $\alpha$ and $\beta \subset s_B$ as in the figure below. If this bigon contains a part of $D$ as in the figure below (right), then there exists at least one crossing for each side of the bigon. Then the boundary of this bigon is a composite circle, so it contradicts our assumption that $D$ is prime. Therefore, the configuration should be as in the figure below (left), so we can take a cutting arc $\delta$ by connecting the two vertices of the bigon as in the figure below (left).

\centerline{\includegraphics[height=0.3in]{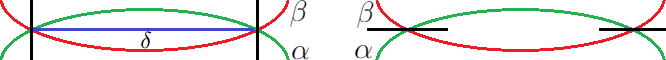}}

Next, we prove that each cutting arc $\delta$ has a corresponding simple loop $\gamma$ on $F(D)$ which satisfies conditions ($2$) and ($3$) of the definition of cutting loops. By definition, two endpoints of $\delta$ lie on $\alpha$, $\alpha \subset s_A$. Connect the two endpoints with an arc $\delta_A$ on the state disk $\overline{\alpha}$ bounded by $\alpha$. As in the proof of Lemma 3.1 in \cite{ADK}, $S^2$, crossing balls, and state disks cut $S^3$ into disjoint balls. Among those balls, we can find a ball whose boundary consists of $\overline{\alpha}$ and a face of $D$ containing $\delta$ inside. Therefore, $\delta \cup \delta_A$ bounds a disc inside that ball. By construction, each ball is contained in one of the handlebodies bounded by $F(D)$, and so does the disc bounded by $\delta \cup \delta_A$. By the same argument, we can find another arc $\delta_B$ in the state disk bounded by $\beta \subset s_B$, and a loop $\delta \cup \delta_B$ which bounds a disk in the same handlebody as $\delta \cup \delta_A$. Then $\gamma = \delta_A \cup \delta_B$ is a simple loop on $F(D)$ which satisfies the conditions ($2$) and ($3$) of the definition of cutting loops.  

Now, we show that $F'(D)=F(D')$. Surgery along $\delta$ divides each state circle into two pieces, and each of them is a state circle of $D'$ because the choice of smoothing did not change. By definition, surgery along $\gamma$ changes $D$ into $D'$. So if we consider a copy of the cobordism between $s_A$ and $s_B$ in $F(D)$, surgery along $\gamma$ changes this cobordism into a cobordism between state circles of $D'$. Moreover, surgery along $\gamma$ divides $\overline{\alpha}$ and $\overline{\beta}$ into two disks each, so each boundary component of the new cobordism is closed up with a disk. Therefore, $F(D')$ is equal to $F'(D)$. See the last figure of Section \ref{definitions}, which describes the cutting loop surgery.

Lastly, we prove that condition ($1$) of the definition of cutting loops holds. If $\gamma$ is separating, then $F'(D)$ is disconnected, which implies that $D'$ is disconnected since $F'(D)=F(D')$. Therefore, surgery along $\delta$ disconnects $D$, which implies that $D$ is not prime. This is a contradiction, so $\gamma$ is non-separating, hence essential. By this non-separating property, $g_T(D')=g_T(D)-1$ is obvious.\end{proof}

\begin{lemma}\label{primelemma}
Any two faces of a prime diagram $D$ can share at most one edge. 
\end{lemma}
\begin{proof}
Two edges determine a  
composite circle, contradicting that $D$ is prime. 
\end{proof}

\begin{proof}[Proof of Theorem \ref{Turaev1}]

\begin{claim}\label{essentialclaim}
A boundary of every face of $D \in S^2$ which contains a non-alternating edge is an essential loop of $F(D)$.
\end{claim}

Note that from the proof of Theorem 3.4 of \cite{ADK}, the boundary of every face can be isotped along $F(D)$ to intersect any other boundary transversally at the midpoints of non-alternating edges of $D$. See figure below:

\centerline{\includegraphics[height=1.15in]{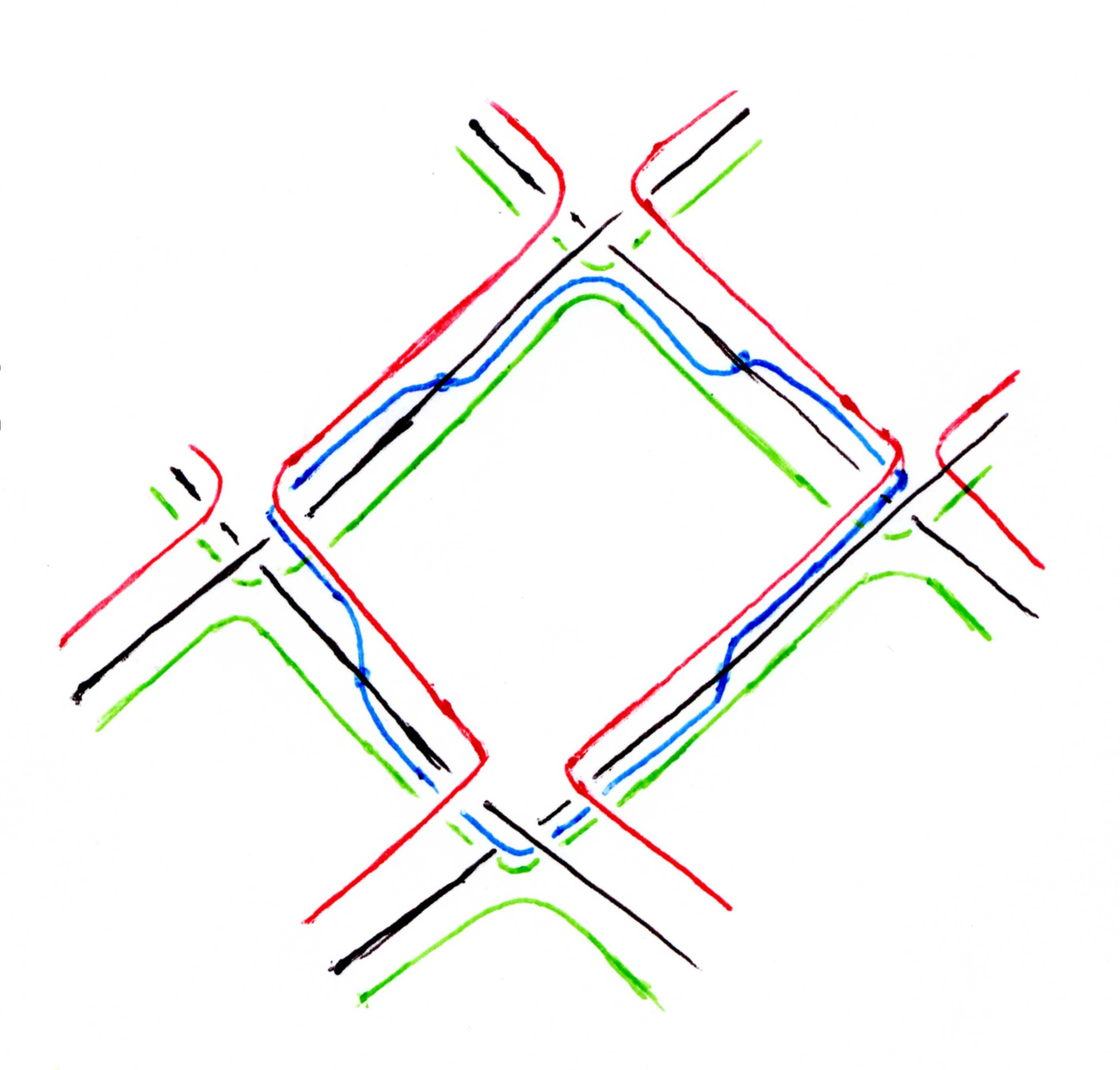}}
Consider a pair of faces which share a non-alternating edge. By Lemma \ref{primelemma}, this is the only edge shared by those two faces. The boundaries of these two faces can be isotoped to intersect only at the midpoint of such a non-alternating edge. Hence, these curves are essential on $F(D)$.

Let $\delta$ be a cutting arc of $D$. Assume that $\delta$ is in a black face $B$ of $D$, and the corresponding cutting loop $\gamma$ is a meridian of $F(D)$. 

\begin{claim}\label{twowhite}
Two white faces of $D$ have non-alternating edges of $D$ on their boundaries.
\end{claim}


By Claim \ref{essentialclaim}, a boundary of every face which contains non-alternating edges is either meridian or longitude. There are only two white faces $W$ and $W'$ which each intersects $\gamma$ once on its boundary. This implies that $\partial W$ and $\partial W'$ are longitudes. Any two faces with the same color are contained in the same handlebody bounded by $F(D)$, so a boundary of every white face is either longitude or trivial on $F(D)$. Every longitude intersect a meridian, so these are the only two white faces which contain non-alternating edges on their boundaries.

\centerline{\includegraphics[scale=0.4]{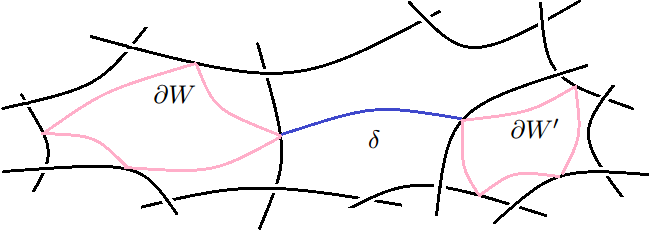}}

Connect every pair of adjacent midpoints of non-alternating edges with a simple arc entirely in a black face. Then all such arcs are parallel to $\delta$ in $S^2 - (W \cup W')$, so they cut $D$ into $2$-tangles (see figure below).
Furthermore, each $2$-tangle is alternating because all edges of the $2$-tangle other than the four half edges are alternating. Hence, $D$ is a cycle of alternating $2$-tangles.
This completes the proof of Theorem \ref{Turaev1}.\end{proof}

\centerline{\includegraphics[scale=0.4]{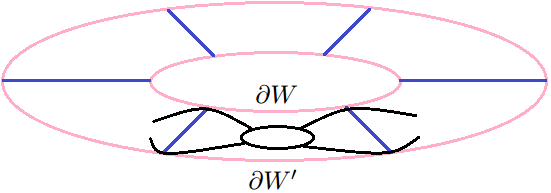}}

\begin{cor}\label{NAprime}
If $D$ is a prime non-alternating diagram on $S^2$, then $g_T (D) > 0$.
\end {cor}

\begin{cor}\label{obstruction}
Let $D$ be a cellularly embedded alternating link diagram on a Heegaard surface $F$ of $S^3$ with $g(F) \geq 1$. If $F$ is the Turaev surface of $L$ which is represented by $D$, then there exists an essential simple loop on $F$ which intersects $D$ twice and bounds a disc in one of the handlebodies bounded by $F$ .
\end{cor}

For example, the figure below (left) is an alternating link diagram on a torus. There is no simple loop on the torus which intersects the link diagram twice. Hence, by Corollary \ref{obstruction}, this link diagram on the torus cannot come from the Turaev surface algorithm.

\centerline{\includegraphics[height=0.5in]{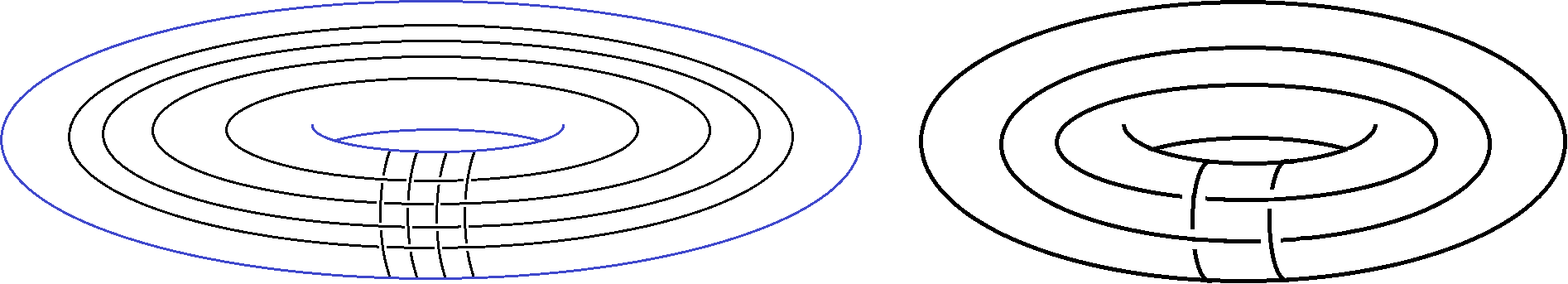}}

In \cite{H1}, Hayashi defined the following complexity of a cellularly embedded, reduced alternating diagram on a closed orientable surface with positive genus. 
\begin{definition}\label{Hayashicomplexity}(\cite{H1})
The complexity $a(F,D)$ of a cellularly embedded, reduced alternating link diagram $D$ on $F$ is defined by

$a(F,D) = \min \{|l \cap D| ; \text{$l$ is an essential simple loop on $F - $\{crossings\}} \}$
\end{definition}

\begin{cor}\label{complexitycor}
Let $D$ be a connected prime non-alternating link diagram on $S^2$. Then $a(F(D),D)=2$. 
\end{cor} 

Note that even if we have a cellularly embedded, reduced alternating diagram $D$ on some Heegaard surface $F$ such that $a(F,D) = 2$, it might not be a Turaev surface. For example, the connected diagram in the figure above (right) has four crossings on $F$, but any connected planar diagram of this split link has more than four crossings. Hence, this link diagram on the torus also cannot come from the Turaev surface algorithm.

\section{Inadequate links with Turaev genus one}
\begin{definition}
A crossing $c$ of $D$ is called an \textit{$A$-loop (resp. $B$-loop) crossing} if it corresponds to a loop of an all-$A$ ribbon graph(resp. all-$B$ ribbon graph) of $D$. We say $c$ is a \textit{loop crossing} if it is an $A$-loop or a $B$-loop crossing. If $c$ is both an $A$-loop crossing and a $B$-loop crossing, then $c$ is called an \textit{$AB$-loop crossing}. If there are no loop crossings, then $D$ is called an \textit{adequate diagram}. A diagram with no $A$-loop or no $B$-loop crossings is called a \textit{semi-adequate diagram}. Otherwise, it is called an \textit{inadequate diagram}. A link is \textit{adequate} if it has an adequate diagram. A link is \textit{semi-adequate} if it has a semi-adequate diagram but does not have an adequate diagram. Otherwise, a link is \textit{inadequate}.
\end{definition}

\begin{lemma}
\label{coreloop}
Let $c$ be a loop crossing and $l(c)$ be a corresponding loop of the ribbon graph. Then a core $\mu$ of $l(c)$ bounds a disc $V$ in one of the handlebodies bounded by $F(D)$. Furthermore, we can perturb $V$ to intersect $S^2$ in a simple arc $\nu$ on $S^2$ such that $\nu \cap D = \partial \nu$. 
\end{lemma}

\begin{proof}
 Both the all-$A$ and all-$B$ ribbon graphs are naturally embedded in $F(D)$, so each core loop is a simple loop on $F(D)$. Then it bounds a disc in one of the handlebodies bounded by $F(D)$. Using the same argument as in Theorem \ref{surgerytheorem}, we can show that $V$ can be isotoped to intersect $S^2$ in a simple arc $\nu$.
\end{proof}

\begin{lemma}
\label{intersectionnumber}
Let $D$ be a prime link diagram with $g_T(D)=1$. Let $l$ be a longitude of $F(D)$. If a cutting loop of $D$ is a meridian of $F(D)$, then $\min|l \cap D| =$ $\#$\{maximally connected alternating tangles of $D$\}. 
\end{lemma}
\begin{proof}
From the cycle of alternating tangle structure of $D$, the link diagram on $F(D)$ is as shown in the figure below.
In this figure, vertical lines correspond to the cutting loops. Then the longitudes are isotopic to the horizontal lines. Each circle represents an alternating 2-tangle, which has at least one crossing inside. Therefore, the horizontal lines minimize the number of intersections. Thus, $min|l \cap D| =$ $\#$\{maximally connected alternating tangle of $D$\}.
\end{proof}

\centerline{\includegraphics[scale=0.2]{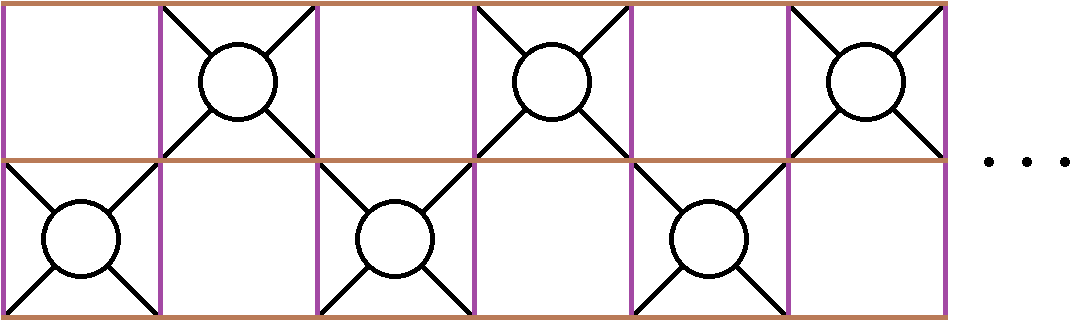}}

\begin{lemma}\label{core-cuttingloop}
Let $D$ be a prime link diagram with $g_T(D)=1$ which is not adequate. Let $c$ be a loop crossing of $D$, and $\mu$ a simple loop on $F(D)$ as in Lemma \ref{coreloop}. Then there exists a cutting loop $\gamma$ of $F(D)$ which is isotopic to $\mu$.    
\end{lemma}
\begin{proof}
From Lemma \ref{coreloop}, $\mu$ is either meridian or longitude. If the number of maximally connected alternating tangles of $D$ is two then we can find a cutting loop which is isotopic to the meridian, and another cutting loop which is isotopic to the longitude. 
If the number of maximally connected alternating tangles of $D$ is greater than two, and if $\mu$ is not isotopic to $\gamma$, then by the Lemma \ref{intersectionnumber}, $|\mu \cap D| > 2$. Therefore $\mu$ is isotopic to $\gamma$. 
\end{proof}

\begin{remark}\label{core-cuttingloopremark}
Lemma \ref{core-cuttingloop} implies that the cutting arc $\delta$ and the simple arc $\nu$ in Lemma \ref{coreloop} are parallel, as in the figure below (left).
In other words, if we surger $D$ along $\nu$, it reduces the Turaev genus of $D$ by one.
\end{remark}

\centerline{\includegraphics[scale=0.3]{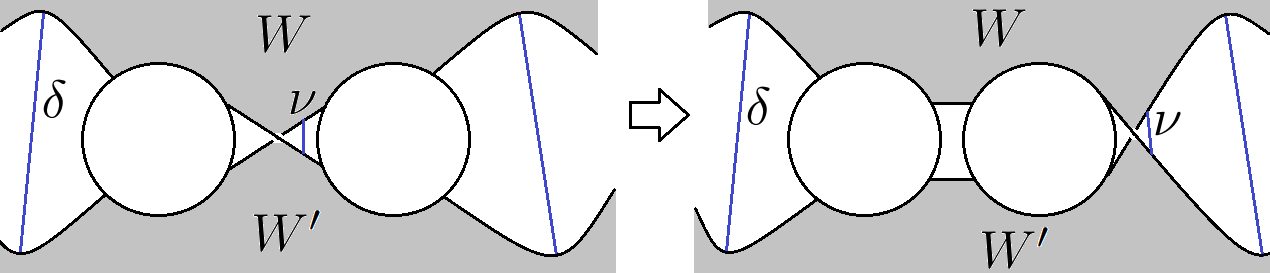}}

\begin{proof}[Proof of the Theorem \ref{almostalternating}]

Let $D_1$ be a prime link diagram of $L$ with $g_T(L)=1$. Assume that $D_1$ has more than two maximally connected alternating $2$-tangles and cutting loops are isotopic to the meridian. By Lemma \ref{core-cuttingloop} and Remark \ref{core-cuttingloopremark}, we can flype $D_1$ as in the figure above to collect all loop crossings into one twist region and reduce all possible pairs of crossings in twist region by Reidemeister-\textrm{II} moves. Note that flypes and Reidemeister-\textrm{II} moves above do not change the Turaev genus. 


If the resulting diagram $D_2$ has more than two maximally connected alternating tangles, then the set of all loop crossings of $D_2$ and the set of crossings in the twist region are the same. Moreover, By Lemma \ref{intersectionnumber}, none of them can be an $AB$-loop crossing. All loop crossings have the same sign, hence, $D_2$ is a semi-adequate diagram, which contradicts our assumption that $L$ is inadequate. Hence, $D_2$ has two maximally connected alternating tangles, so there are two non-isotopic cutting loops. Therefore, $D_2$ can have loop crossings which are not in the twist region above. Then without loss of generality, the configuration of $D_2$ is one of the figures in the figure below (left), in which the crossings in the figures are possible loop crossings. Then we can see from the figure below (right) that $D_2$ has $B$-loop crossings if and only if one of the maximally connected alternating $2$-tangles contains only one crossing. Therefore, $D_2$ is an almost-alternating diagram.
\end{proof}
\centerline{\includegraphics[scale=0.2]{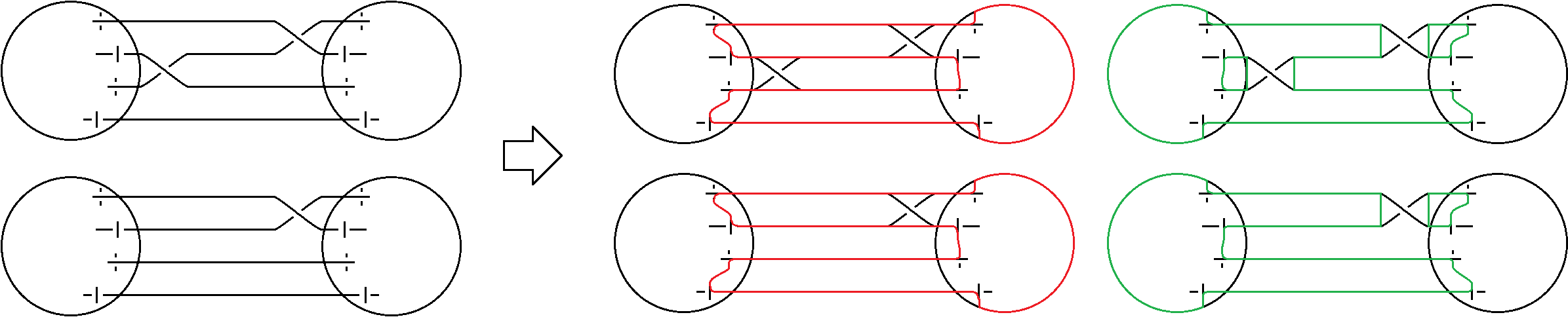}}
\begin{cor}
Let $D$ be a reduced Turaev genus one diagram of a trivial knot. Then there exists a sequence of Turaev genus one diagrams
$$D = D_1 \rightarrow \ldots D_k = D'_1 \rightarrow \dots D'_l = C_m$$
which satisfy the following :
\begin{enumerate}
\item $D_{i+1}$ is obtained from $D_i$ by a flype or a Reidemeister II-move,
\item Each $D'_i$ is almost-alternating,
\item $D'_{i+1}$ is obtained from $D'_i$ by a flype, an untongue or an untwirl move.
\item $C_m$ is the diagram in Figure below (c).
\end{enumerate}
\end{cor}

\begin{proof}
Every reduced diagram of a trivial knot is inadequate. 
The proof of Theorem \ref{almostalternating} implies that every reduced prime diagram $D$ of the trivial knot with $g_T(D)=1$ can be changed to an almost-alternating diagram by flypes and Reidemeister II-moves. In Theorem $5$ of \cite{Ts1}, Tsukamoto proved that every almost-alternating diagram of the trivial knot can be changed to $C_m$ by flypes, untongue moves and untiwrl moves via a sequence of almost-alternating diagrams.
\end{proof}

\centerline{\includegraphics[scale=0.3]{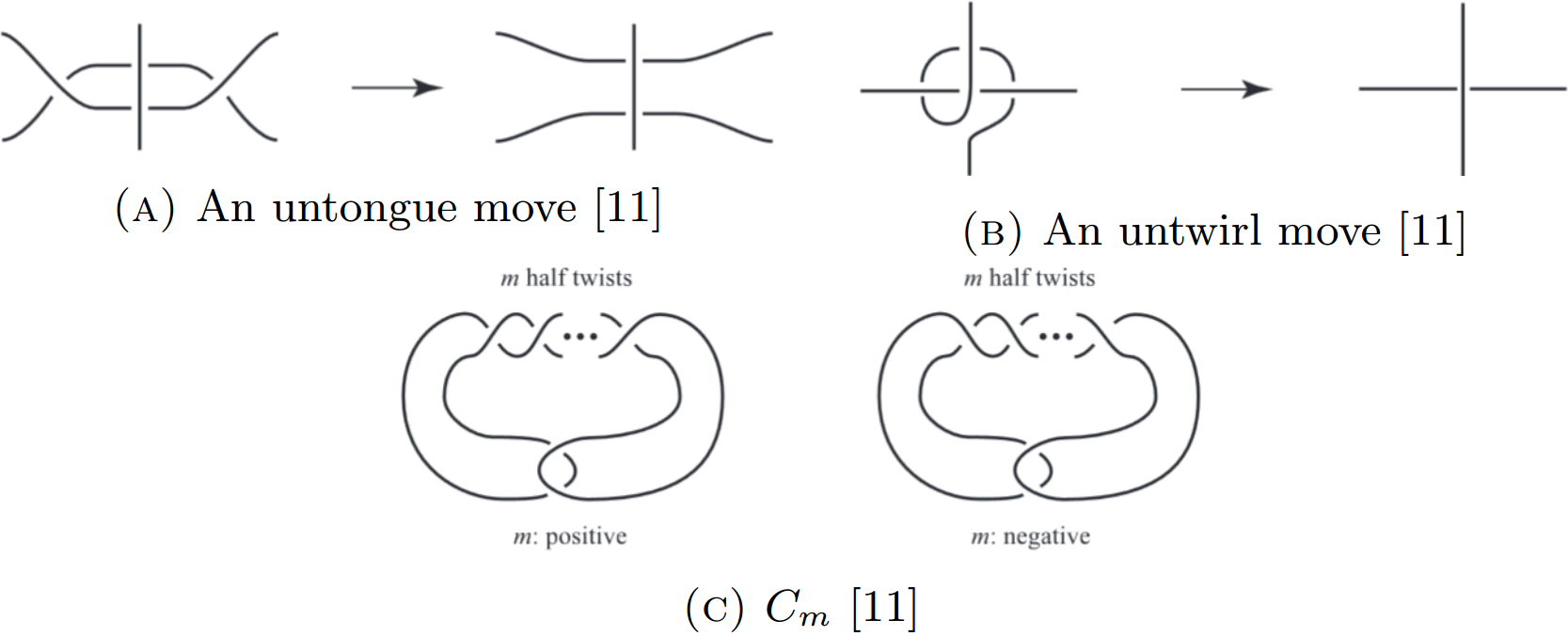}}

\section{Classification of Turaev genus two diagrams}
A set of disjoint simple loops on $S^2$ is said to be \textit{concentric} if the annular region on $S^2$ bounded by any two curves does not contain a curve which bounds a disc inside the region. 

\begin{theorem}\label{concentric}
Let $\delta$ be a cutting arc of a prime non-alternating diagram $D$ with $g_T(D)=g$. Assume that $\delta$ is in a black face of $D$. If we surger $D$ along $\delta$ to get $D_1$, then $D_1$ satisfies the following :\begin{enumerate}

\item The composite circles of $D_1$ are concentric. 
\item Let $D_2$ be a link diagram obtained from $D_1$ by surgery along every arc which is the intersection of a black face and a composite circle of $D_1$. Then each component of $D_2$ is prime and the sum of Turaev genera of all components is $g - 1$. 
\end{enumerate}
\end{theorem}

\begin{proof}
Let $B$ be a black face of $D$ which contains $\delta$. Let $W$ and $W'$ be white faces of $D$ such that $\partial \delta \cap \partial W \neq \emptyset$ and $\partial \delta \cap \partial W' \neq \emptyset$. Surgery along $\delta$ joins $W$ and $W'$ into $W_1$ and divides $B$ into $B_1$ and $B'_1$ (see figure below).
Every other face of $D$ is not changed by surgery, so it is a face of $D_1$ as well. 
\begin{claim} \label{W_1}
Every composite circle of $D_1$ intersects $W_1$.
\end{claim}
Assume there exists a composite circle of $D_1$ which does not intersect $W_1$. Then there exists a different white face $W'_1$ of $D_1$ which shares two edges with some black face $B'$ of $D_1$. $W'_1$ can be considered as a white face of $D$ and by Lemma \ref{primelemma}, $W'_1$ shares only one edge with other black faces of $D$, so $B'$ is a join of two black faces of $D$. However, surgery along $\delta$ cannot join two black faces, which is a contradiction.
\begin{claim}\label{black1} 
Every black face of $D_1$ intersects at most one composite circle of $D_1$.
\end{claim}
By Claim \ref{W_1}, every black face which intersects composite circles is adjacent to $W_1$. Every black face of $D$ except $B$ is not changed by surgery, so Lemma \ref{primelemma} implies each black face intersects at most one composite circle. Now, $B$ shares one edge each with $W$ and $W'$. After surgery, those two edges are changed to two edges $e$ and $e'$ in $D_1$, each on the boundary of different black faces. Therefore, $W_1$ shares one edge each with $B_1$ and $B'_1$, so $B_1$ and $B'_1$ do not intersect with any composite circle of $D_1$.  See the figure below. 

\centerline{\includegraphics[scale=0.3]{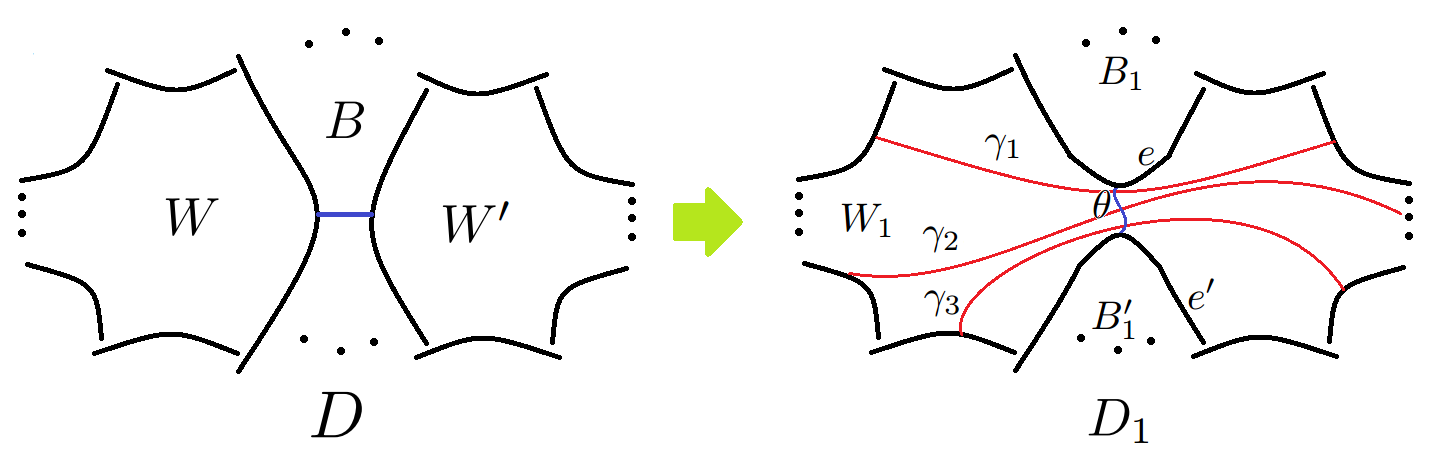}}
\begin{claim}\let\qed\relax
The composite circles of $D_1$ are concentric.
\end{claim}
Let $\{\gamma_i\}$ be the set of composite circles of $D_1$. By Claim \ref{black1}, $\forall i,j$, $i \neq j$, $\gamma_i \cap \gamma_j \subset W_1$. 
The number of intersections is even, so we can remove all intersections by perturbing composite circles inside $W_1$.
 By Lemma \ref{primelemma}, $\partial \gamma_i \cap W_1$ consists of two points one in $\partial W$ and another in $\partial W'$. By the proof of Claim \ref{black1}, $\gamma_i \cap e = \gamma_i \cap e' = \emptyset$. Then we can connect midpoints of $e$ and $e'$ with a simple arc $\theta$ such that $|\theta \cap \gamma_i| = 1, \forall i$ (see figure above).
 If the composite circles are not concentric, then there exists a triple $(\gamma_1, \gamma_2, \gamma_3)$ such that $\gamma_2$ bounds a disc inside an annulus on $S^2$ bounded by $\gamma_1$ and $\gamma_3$. 
 Then 
 $\theta$ intersects $\gamma_2$ an even number of times, which is a contradiction. 
Now we will complete the proof by showing (2). The sum of Turaev genera of all components of $D_2$ is $g_T(D)-1$ by Theorem \ref{surgerytheorem} and additivity of Turaev genera of diagrams under connected sum. Assume that one of the components of $D_2$ is composite. Suppose $W_1$ is changed to $W_2$, which is homeomorphic to an $n$-holed disc after surgery. By the same argument as in Claim \ref{black1}, every black face of $D_1$ which intersects composite circles of $D_1$ is divided into two faces and each face shares exactly one edge which apears after surgery with $W_2$. Therefore, every composite circle of $D_2$ intersects with edges of $D_1$. Now consider each composite circle as a union of two arcs, each of them intersects a face of $D_2$. Using the checkerboard coloring of $D_2$, we can label each arc as a black or white arc. Every face of $D_2$ except $W_2$ is a subset of a face of $D_1$. Therefore, every black and white arc except the one inside $W_2$ is a simple arc inside a face of $D_1$. For the white arc inside $W_2$, we can choose another arc with the same endpoints, which is an arc inside $W_1$ because its endpoints are on the edges of $D_1$. Then the black and white arcs form a composite circle of $D_1$, which contradicts our assumption that we surgered along all composite circles to get $D_2$. This completes the proof of Theorem \ref{concentric}. 
\end{proof}

\begin{proof}[Proof of Theorem \ref{Turaev2}]
Let $D$ be a prime link diagram on $S^2$ with $g_T(D) = 2$. Choose a cutting arc $\delta$ using an algorithm from the proof of Theorem \ref{surgerytheorem} and assume that $\delta$ is in a black face of $D$. We surger $D$ along $\delta$ to get $D_1$ with $g_T(D_1)=1$. $D_1$ has a checkerboard coloring induced by the checkerboard coloring of $D$. 

Let $D'$ be obtained from $D$ by surgery along an arc $\tau$. We define the \textit{attaching edge} $\tau'$ to be midpoint$(\tau)\times [ - \epsilon, \epsilon ]$, with $(\tau, \epsilon)$ as in the definition of surgery along a cutting arc, as indicated by a dotted arc. Note that if we do surgery along $\tau'$, then the attaching edge is $\tau$, and we get $D$ again.

Consider every composite circle of $D_1$. 
We surger $D_1$ along black arcs to get $D_2$ which consists of exactly one prime diagram $T$ with $g_T(T)=1$, and several prime alternating diagrams. Choose the checkerboard coloring of $T$ that comes from $D$. Note that every attaching edge is in one white face of $T$, as in the figure below.

\centerline{\includegraphics[scale=0.15]{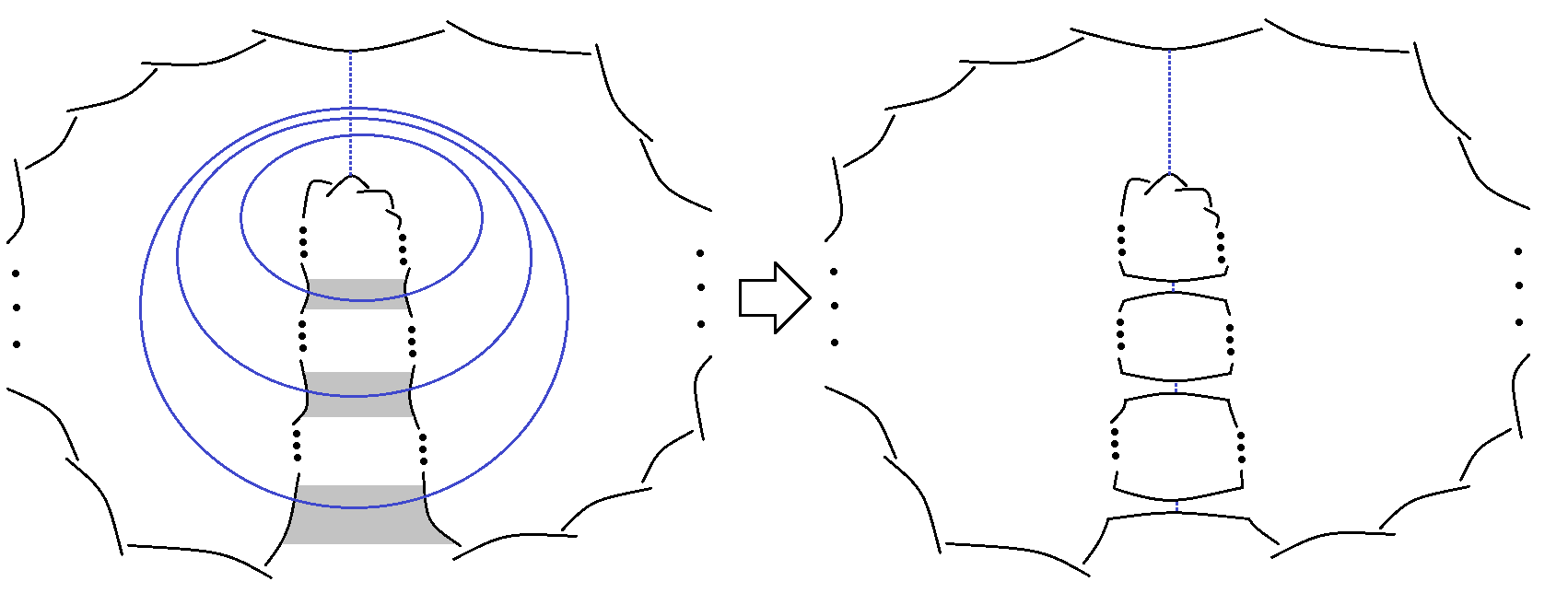}}

Now we need to reconstruct $D$ from $T$ and the alternating diagrams. Theorem \ref{concentric} implies components of $D_2$ are pairwise connected by exactly one attaching edge, if any, and no more than two attaching edges in total. Below, we consider all possible cases for attaching $T$ and the alternating components of $D_2$:

\begin{case*}\textit{Every cutting arc of $T$ is inside a black face of $T$.}
\end{case*}
Every other component of $D_2$ is inside a white face $W$ of $T$, so we have four different sub-cases.

\lowerromannumeral{1}) $W$ has non-alternating edges on its boundary.
See the figure below, where $W$ is the yellow face shown:

\centerline{\includegraphics[scale=0.2]{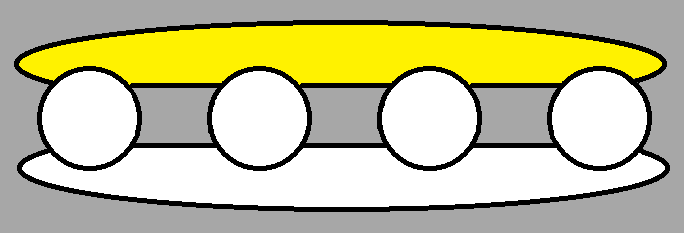}}

If two attaching edges are connected to two alternating edges of the same alternating tangle of $T$, then we have an alternating $4$-tangle, and the alternating tangle structure of $D$ is shown in Figure \ref{case1}.
If the two attaching edges are connected to two alternating edges in different alternating tangles of $T$, then we have two alternating $3$-tangles, and the alternating tangle structure of $D$ is shown in Figure \ref{case2}.
If one of the attaching edges are connected to a non-alternating edge of $\partial W \subset T$, then the sign of crossings of such an alternating diagram is the same as one of the alternating tangles adjacent to such a non-alternating edge. Hence, we can merge the alternating tangles, as shown in the figure below. Therefore, in this case, the alternating tangle structure is the same as one of the above cases. 
\centerline{\includegraphics[scale=0.15]{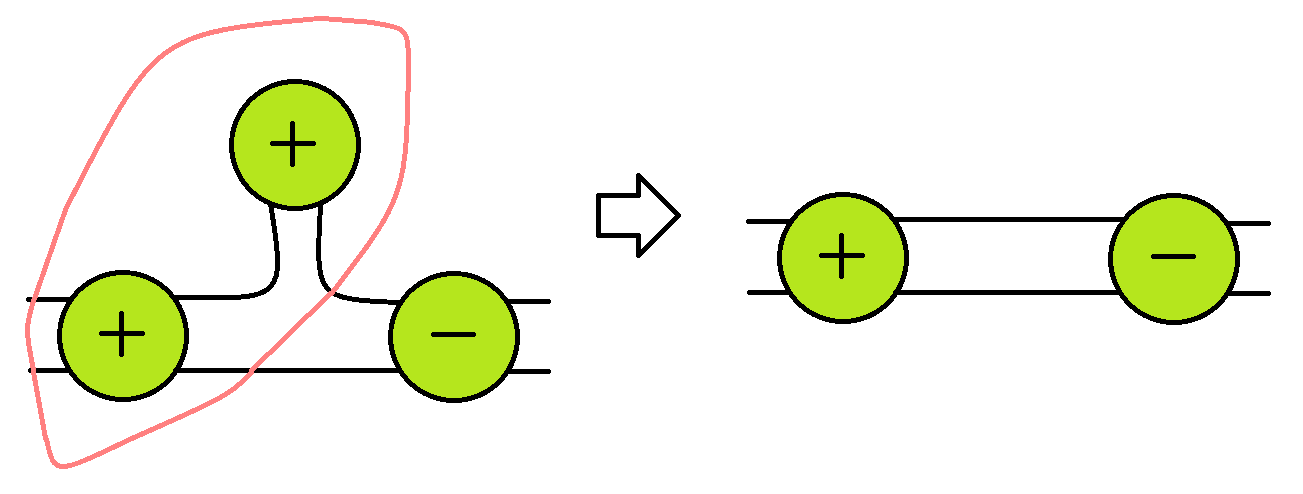}}

\lowerromannumeral{2}) $W$ is contained in one of the alternating tangles, and $W$ is adjacent to a black face $B$ which has a cutting arc inside, as in the figure below:

\centerline{\includegraphics[scale=0.2]{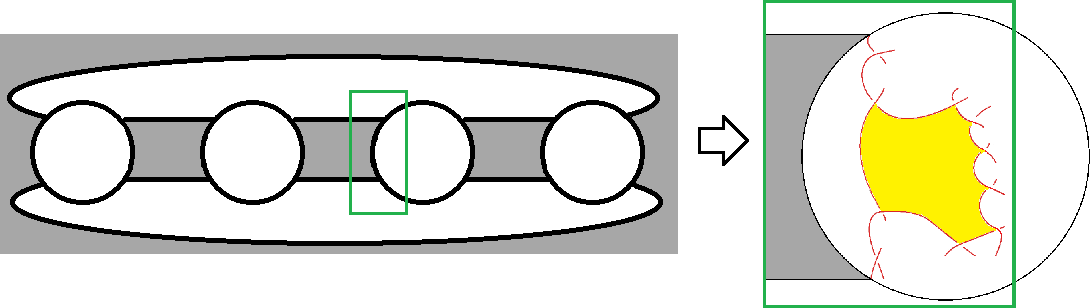}}

If one of the attaching edges is connected to $\partial B$, then we have two possibilities. First, if the sign of the alternating tangle of $T$ and of the alternating diagram are different, then the alternating tangle structure changes as illustrated in the figure below. 
Then we have one alternating $4$-tangle, and the alternating tangle structure of $D$ is shown in Figure \ref{case3}.
If the signs are the same, then we have one alternating $4$-tangle which is not simply connected, and the alternating tangle structure is shown in Figure \ref{case4}.
If there is no attaching edge connected to $\partial B$, then the alternating tangle structure is the same as Figure \ref{case4}.

\centerline{\includegraphics[scale=0.2]{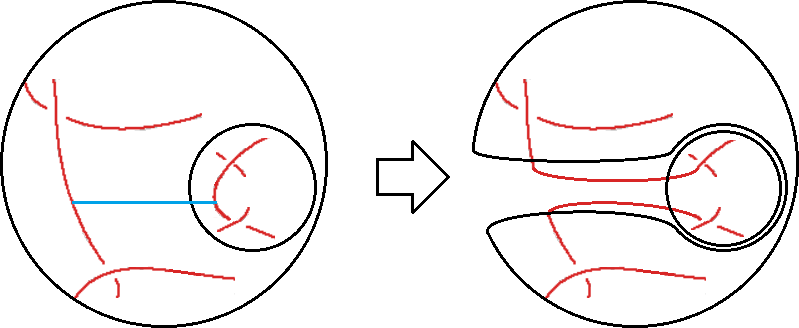}}

\lowerromannumeral{3}) $W$ is contained in one of the alternating tangles, and adjacent to black faces $B$ and $B'$ which each have a cutting arc inside, as in the figure below:

\centerline{\includegraphics[scale=0.2]{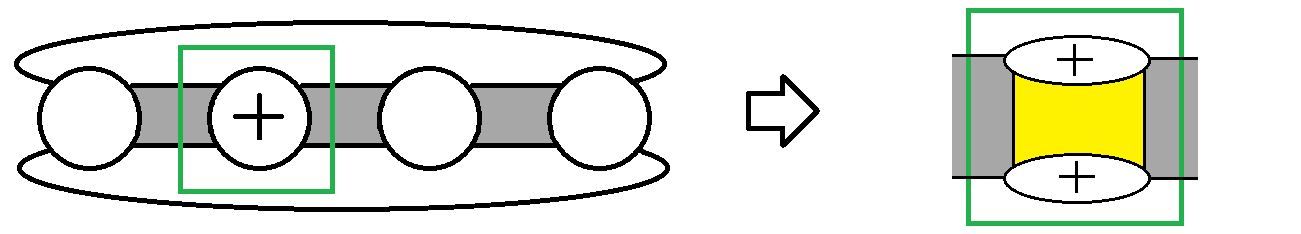}}

If one attaching edge is connected to $\partial B$ and another attaching is connected to $\partial B'$, then we have three possibilities. First, if the sign of the alternating tangle of $T$ and of two alternating diagrams connected to $T$ by two attaching edges are different, then the alternating tangle structure changes as in the figure below (left).
Therefore, every maximally connected alternating tangle is a $2$-tangle, and the alternating tangle structure of $D$ is shown in Figure \ref{case8}.

\centerline{\includegraphics[scale=0.2]{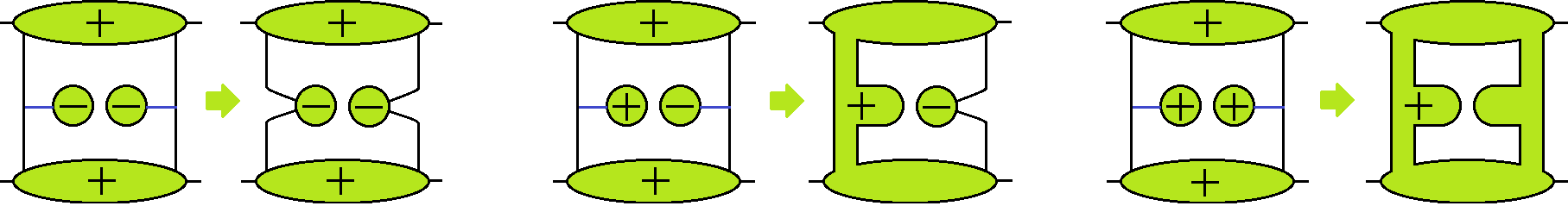}}
If the sign of one of the alternating diagrams is the same as the sign of the alternating tangle of $T$, then we can merge them into one maximally connected alternating tangle as in the figure above (middle). Then we have one alternating $4$-tangle, and the alternating tangle structure of $D$ is shown in Figure \ref{case3}.

If the signs of two alternating diagrams are the same as the sign of the alternating tangle of $T$, then we can merge them into one maximally connected alternating tangle as in the figure above (right). This maximally connected alternating tangle is not simply connected and the alternating tangle structure of $D$ is shown in Figure \ref{case4}.
Other cases are just the same as case \lowerromannumeral{2}) above. 

\lowerromannumeral{4}) A black face adjacent to $W$ cannot have non-alternating edges on its boundary. This case is the alternating tangle structure shown in Figure \ref{case4}.

\begin{case*}\textit{Every cutting arc of $T$ is inside a white face of $T$}
\end{case*}
\lowerromannumeral{1}) $W$ contains a cutting arc of $T$, as in 
the figure below:

\centerline{{\includegraphics[scale=0.2]{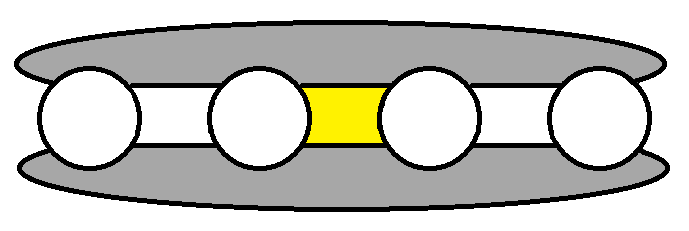}}}

If two attaching edges are connected to alternating edges of $T$, and the two alternating edges of $T$ are in different tangles, then we have two alternating $3$-tangles and the alternating tangle structure is shown in Figure \ref{case5}.
If two attaching edges are connected to alternating edges of $T$, and the two alternating edges of $T$ are in the same alternating tangle, then we have one alternating $4$-tangle and the alternating tangle structure is shown in Figure \ref{case6}. 
If at least one attaching edge is connected to a non-alternating edge of $T$, then the alternating tangle structure changes as in the figure in the proof of Case(1\lowerromannumeral{1}), which implies the same alternating tangle structure as in Figrue \ref{case5} or \ref{case6}. 

\lowerromannumeral{2}) $W$ does not contain a cutting arc, but is adjacent to two black faces $B$ and $B'$ which have non-alternating edges on their boundaries, as in the figure below:

\centerline{\includegraphics[scale=0.2]{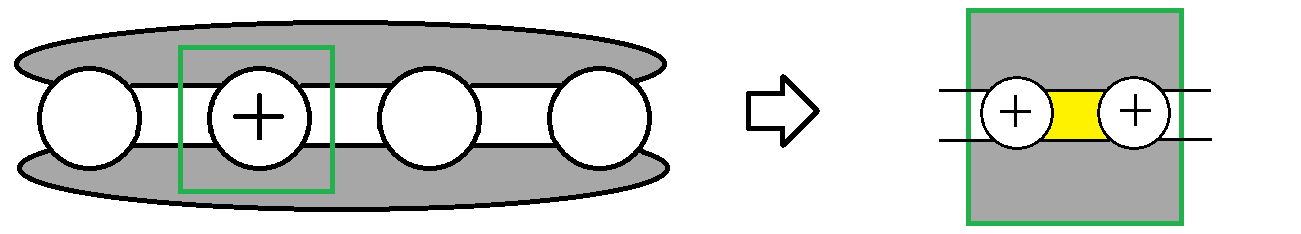}}

Assume that the two alternating tangles adjacent to $W$ are positive tangles, as in the figure above. 
If two attaching edges are not connected to the edges of $\partial B$ nor $\partial B'$ then the alternating tangle structure is the same as in Figure \ref{case4}.  
If exactly one attaching edge is connected to an edge of either $B$ or $B'$, and an alternating diagram attached to it has negative crossings, then the alternating tangle structure changes as in the figure in the proof of Case(1\lowerromannumeral{2}). Therefore, we have one alternating $4$-tangle and the alternating tangle structure is shown in Figure \ref{case6}. If the alternating diagram has positive crossings, then the alternating diagram and the alternating tangle of $T$ merge.
Therefore, it has the same alternating tangle structure as in Figure \ref{case4}.
If two attaching edges are connected to the edges of $B$ and $B'$, and both alternating diagrams attached to $T$ along them have negative crossings, then the alternating tangle structure changes as in left figure in the proof of Case(1\lowerromannumeral{3}). Therefore, every alternating tangle of $D$ is a $2$-tangle, and the alternating tangle structure is shown in Figure \ref{case7}.
Otherwise, the alternating tangle structure of $D$ can be as in Figure \ref{case4} or Figure \ref{case6}.

\lowerromannumeral{3}) $W$ is adjacent to exactly one black face $B$ which has non-alternating edges on its boundary as in the figure below:

\centerline{\includegraphics[scale=0.2]{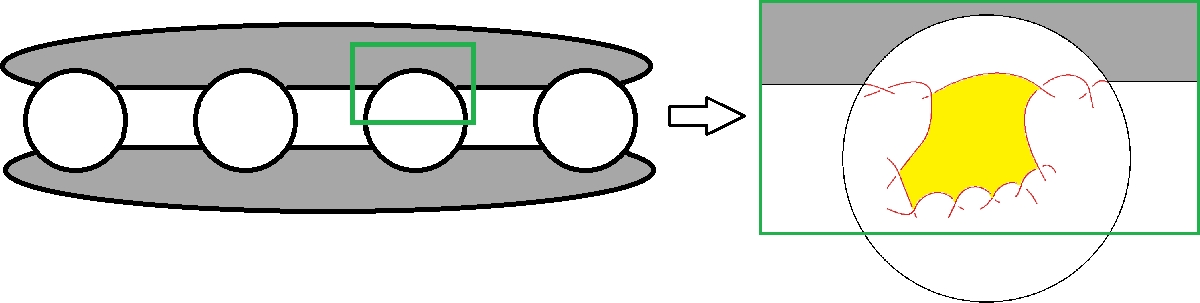}}
If two attaching edges are not connected to $\partial B$, then the alternating tangle structure is shown in Figure \ref{case4}.
If one attaching edge is connected to $\partial B$, then it is as shown in Figure \ref{case4} or Figure \ref{case6}, depending on the sign of the alternating tangle attached to that attaching edge.

\lowerromannumeral{4}) A black face adjacent to $W$ cannot have non-alternating edges on its boundary. This is same case as 1(\lowerromannumeral{4}), which is the alternating tangle structure in Figure~\ref{case4}.

To show that we have considered all the possible cases, we need to show all faces of $T$ are used in the proof. First, all faces of $D$ in the channel region are considered in Case(1\lowerromannumeral{1}) and Case(2\lowerromannumeral{1}). It remains to show that all the faces in the alternating tangles are used in the proof. From the checkerboard coloring and the cycle of alternating $2$-tangle structure, we can show that every face in the alternating tangle can be adjacent to at most two faces in the channel region. Therefore we can categorize every faces in the alternating tangle by the number of adjacent faces in the channel region and  the existence of cutting arcs in adjacent faces. These are considered in the Cases (1\lowerromannumeral{2} - 1\lowerromannumeral{4}) and Cases (2\lowerromannumeral{2} - 2\lowerromannumeral{4}).

Lastly, we show that all eight cases are distinct up to isotopy on $S^2$. First, Case 4 is distinct from all others because it has a non-simply connected alternating tangle. If every ribbon contains no alternating tangles, then Cases 1, 3 and 6 have the same alternating tangle structure. Similarly, Cases 2 and 5 have the same alternating tangle structure. Cases 1,3,6 have a $4$-tangle, and Cases 2,5 have two $3$-tangles, so they are distinct. Cases 7 and 8 are distinct from the others because their alternating tangle structure consists of only $2$-tangles. Case 8 has $2$-tangles adjacent to four others which Case 7 does not, so Cases 7 and 8 are distinct.  
We now distinguish Cases 1, 3 and 6. With many alternating tangles in every ribbon, the single $4$-tangle is connected to four different alternating $2$-tangles. If we orient the boundary of the $4$-tangle, non-alternating edges connected to the boundary have a cyclic ordering. If we compare the three cyclic orderings, then they are distinct up to a cyclic permutation. Therefore, Cases 1, 3 and 6 are all distinct. Similarly, Cases 2 and 5 are distinct. This completes the proof of Theorem \ref{Turaev2}.
\end{proof}

\end{document}